\documentclass[letterpaper, 11pt,  reqno]{amsart}

\usepackage{amsmath,amssymb,amscd,amsthm,amsxtra, esint}

\usepackage[implicit=true]{hyperref}

\usepackage[left=32mm, right=32mm, 
bottom=27mm]{geometry}

\usepackage{mathtools}
\usepackage{enumerate}

\newtheorem{Theorem}{Theorem}
\newtheorem{Corollary}[Theorem]{Corollary}
\newtheorem{Lemma}[Theorem]{Lemma}
\newtheorem{Proposition}[Theorem]{Proposition}

\theoremstyle{definition}

\newtheorem{Remark}[Theorem]{Remark}
\numberwithin{Theorem}{section}

\numberwithin{equation}{section}
\newenvironment{renumerate}{\begin{enumerate}[\normalfont{(}i\normalfont{)}]}{\end{enumerate}}                                                                              

\newcommand{\norm}[1]{\left\|{#1}\right\|}
\newcommand{\wick}[1]{:\hspace{-2pt}{#1}\hspace{-2pt}:}
\newcommand{\jap}[1]{\left\langle{#1}\right\rangle}

\renewcommand{\epsilon}{\varepsilon}
\newcommand{\T}{\mathbb T}
\renewcommand{\H}{\mathcal H}
\newcommand{\N}{\mathbb N}
\renewcommand{\P}{\mathbb P}
\newcommand{\E}{\mathbb E}
\newcommand{\Z}{\mathbb{Z}}
\newcommand{\R}{\mathbb R}
\newcommand{\C}{\mathbb C}
\renewcommand{\d}{\mathrm d}

\newcommand{\loc}{\mathrm{loc}}

\DeclareMathOperator{\supp}{supp}

\newcommand{\dual}[2]{\left\langle#1,#2\right\rangle}

\title[GWP for Stochastic Wave Equation on $\R^2$]{Global well-posedness 
of the two-dimensional
stochastic nonlinear wave equation 
on an unbounded
domain}
\author{Leonardo Tolomeo}

\address{
	Leonardo Tolomeo\\
	Mathematical Institute\\ Hausdorff Center for Mathematics \\
	Universit\"at Bonn\\
	Bonn \\
	Germany} 
\email{tolomeo@math.uni-bonn.de}

\keywords{stochastic nonlinear wave equation; nonlinear wave equation; renormalization; Wick ordering; Hermite polynomial; white noise.}
\subjclass[2010]{35L71, 60H15}

\begin{document}

\maketitle

\begin{abstract}
We study the two-dimensional wave equation with cubic nonlinearity posed on $\R^2$, with space-time white noise forcing. After a suitable renormalisation of the nonlinearity, we prove global well-posedness for this equation for initial data in $\H^s$, $s>\frac45$. 
\end{abstract}
\section{Introduction}
We consider the following stochastic nonlinear wave equation (SNLW) with additive space-time white noise forcing 
\begin{equation} \label{SNLW0}
\begin{cases} 
u_{tt} = \Delta u -{u}^3 + 3\infty\cdot u + \d W, \\
u(0,\cdot) = u_0 \in H^s_\loc, \\
u_t(0,\cdot) = u_1 \in H^{s-1}_\loc,
\end{cases}
\end{equation}
where $\d W$ denotes a space-time white noise on $\R^2$. The term ``$-3\infty\cdot u$" in the equation denotes a time-dependent renormalisation, which has been first introduced by Gubinelli, Koch and Oh in \cite{gko18} for a family of stochastic wave equations with power nonlinearities. In this work, the authors prove local well-posedness for renormalised stochastic wave equations posed on $\T^2$ with certain polynomial nonlinearities.

The necessity to apply  such a renormalisation in oder to get nontrivial solutions was highlighted in a series of works by Albeverio, Haba, Oberguggenberger, and Russo \cite{ahr96,or98,or01,r83}, and more recently by Oh, Okamoto and Robert \cite{oor19}, who showed that without the renormalisation term, solutions to \eqref{SNLW0} must satisfy a linear wave equation.

The renormalisation that we apply in \eqref{SNLW0} is better described as follows. Recall that $\d W$ is defined to be a distribution-valued random
variable such that for every test function $\phi$, $\dual{\phi}{\d W}$ is a gaussian random variable with mean $0$ and variance 
\begin{equation}\label{universalProperty}
\E[|\dual{\phi}{\d W}|^2] = \norm{\phi}_{L^2}^2.
\end{equation}
Ignoring the term with $\infty$ in \eqref{SNLW0}, we consider a perturbative expansion $u = v + \psi$, where 
\begin{equation}
\psi := \int_0^t \frac{\sin((t-t')|\nabla|)}{|\nabla|} \d W(t')
\end{equation}
solves the linear wave equation
$$\psi_{tt} = \Delta \psi + \d W. $$
Formally, the term $v$ would then solve the equation
$$v_{tt} - \Delta v = - (\psi + v)^3 = - \psi^3 - 3 \psi^2 v - 3 \psi v^2 - v^3. $$
However, because of the roughness of $\d W$, it can be shown that the terms $\psi^3$, $\psi^2$ do not make sense as space-time distributions. Therefore, we introduce the Wick renormalisation 
\begin{equation} \label{renormalisation}
\begin{aligned}
\wick{\psi^2}\,&= \psi^2 - \E[|\psi|^2], \\
\wick{\psi^3}\,&= \psi^3 - 3\E[|\psi|^2] \psi,
\end{aligned}
\end{equation}
and \emph{define} $v = u - \psi$ to solve the equation
\begin{equation} \label{SNLWv}
\begin{cases}
v_{tt} - \Delta v = - \wick{\psi^3} - 3 \wick{\psi^2} v^2 - 3 \psi v - v^3 \\
v(0,\cdot) = u_0 \in H^s_\loc(\R^2), \\
v_t(0,\cdot) = u_1 \in H^{s-1}_\loc(\R^2).
\end{cases}
\end{equation}
While both terms on the right hand side of \eqref{renormalisation} diverge (for both definitions), it is actually possible to give a meaning to the renormalised terms $\wick{\psi^2}, \wick{\psi^3}$ by first taking a smooth approximation of the noise $\d W$ and then taking a limit in the space $W^{-\epsilon,\infty}_{\loc}$. This will be carried out explicitly in Section 2. 
Denoting (formally) $ \wick{u^3}\,= u^3 - 3 \E[|\psi|^2] u$, solving the equation \eqref{SNLWv} for $v$ corresponds to solving the equation 
\begin{equation} \label{SNLW} \tag{SNLW}
\begin{cases} 
u_{tt} = \Delta u -\wick{u}^3 + \d W, \\
u(0,\cdot) = u_0 \in H^s_\loc, \\
u_t(0,\cdot) = u_1 \in H^{s-1}_\loc
\end{cases}
\end{equation}
for $u$. Since $\E[|\psi|^2] = + \infty$ for every $t>0$, by inserting this into \eqref{SNLW} we obtain the formula \eqref{SNLW0}. This kind of renormalisation is exactly the same that appears in \cite{gko18} for the cubic wave equation on the torus.
%\subsection{Mild formulation and definition of the flow}
%Consider the \emph{linear} wave equation with forcing $f$ and initial data $\u_0 = \vec{u_0}{u_1}$
%\begin{equation}
%u_{tt} = \Delta u + f, \\
%\end{equation}
%By the variation of constants formula, the solution $u$ of this equation will be given by the formula 
%\begin{equation}
%u = \cos(t|\nabla|)u_0 + \frac{\sin(t|\nabla|)}{|\nabla|}u_1 + \int_0^t \frac{\sin((t-t')|\nabla|)}{|\nabla|} f(t')\d t'. 
%\end{equation}
%We define $\psi$ to be the \emph{stochastic convolution} 
%\begin{equation} \psi = \int_0^t \frac{\sin((t-t')|\nabla|)}{|\nabla|} \xi(t')\d t'.
%\end{equation}

Before stating our main result, we need to define what we mean by solutions of \eqref{SNLW0}. As we already discussed, we write $u = \psi + v$, and we require $v$ to solve the mild formulation of \eqref{SNLWv},
\begin{multline} \label{SNLWvMild}
v = \cos(t|\nabla|)u_0 + \frac{\sin(t|\nabla|)}{|\nabla|}u_1 \\
+ \int_0^t \frac{\sin((t-t')|\nabla|)}{|\nabla|} \Big(- \wick{\psi^3}(t') - 3 \wick{\psi^2}(t') v(t') - 3 \psi v^2(t') - v^3(t')\Big)\d t'.
\end{multline}

\begin{Theorem}
\label{main}
Consider the equation \eqref{SNLW} on $\R^2$, and let $1>s>\frac45$. Then \eqref{SNLW} is almost surely globally well-posed. 
More precisely, for every $(u_0,u_1) \in \H^s_\loc$, there exists a unique $\psi + C(\R;H^s_\loc(\R^2))$-valued random variable $u$ such that almost surely
\begin{itemize}
\item for every stopping time $T>0$, $u|_{[-T,T]}$ is the unique solution to \eqref{SNLW} in the space $\psi + C([-T,T];H^s_\loc)$,
\item for every time $0<T<\infty$, $u$ is continuous in the initial data $(u_0,u_1)$.
\end{itemize}
\end{Theorem}

In the recent years, there have been many developments in the study of global solutions for parabolic stochastic SPDEs, both on a compact domain (see for instance \cite{h18,mw15,mw17} for a study of $\Phi^4_d$ models), and more recently, on the euclidian space (\cite{cmw19,gh19}). 
However, the dispersive nature of the wave equation does not allow to extend the techniques developed for the study of stochastic quantisation equation to the study of equation \ref{SNLW0}.

On the other hand, there are not many global well-posedness results available for stochastic dispersive equations, and this result is the first one the author is aware of that deals with a non compact domain. In \cite{mptw}, the author and his collaborators showed global well-posedness for the nonlinear stochastic beam equation posed on $\T^3$. In \cite{gkot}, the author together with the authors of \cite{gko18}, showed a similar result to Theorem \ref{main}, by extending the solutions of the stochastic wave equation with cubic nonlinearity posed on $\T^2$ built in \cite{gko18} for infinite time. The main difference between this result and Theorem \ref{main} is that we do not have a local well-posedness result available for \eqref{SNLW}. Indeed, due to the unboundedness of the domain, one can show that for every $t>0$, the stochastic convolution $\psi(t)$ satisfies $\norm{\psi(t)}_{X} = +\infty$ 
for every translation invariant norm $\norm{\cdot}_X$ defined on a subspace $X$ of distributions. 
This in turn implies that given $\sigma \in \R$, $t,R>0$, a typical solution of \eqref{SNLWvMild} could satisfy \emph{a priori} $\norm{u(t)}_{H^\sigma(B)} \ge R$ on some ball $B$. Therefore, any perturbative argument for local well-posedness (such as a Banach fixed point argument) is bound to fail.

We therefore follow a different strategy. We take a cutoff function $\rho$ with compact support with $\rho \equiv 1$ on a big ball $B=B(0,R)$,
and consider a localised version of the equation for $v$ given by
\begin{equation} \label{SNLWloc}
\tilde v_{tt}  - \Delta \tilde v + \tilde v^3 + 3\tilde v^2 \rho\psi + 3\tilde v \rho \wick{\psi^2} + \rho \wick{\psi^3}\,\, = 0.
\end{equation}
Notice that in the ball $B$, this equation is exactly the same as \eqref{SNLWv}. Because of finite speed of propagation, we expect $\tilde v(t) = v(t)$ in the ball $B(0,R-t)$\footnote{This will be made more rigorous in Section 5.}. Therefore, we reduce the problem of showing global well-posedness for
\eqref{SNLW}, to showing global well-posedness for \eqref{SNLWloc}, independently of the cutoff $\rho$. 

In order to do so, we show an energy estimate. We would like to estimate the functional 
$$E(\tilde v(t), \partial_t \tilde v(t)) := \frac 12 \int |\partial_t \tilde v|^2 + \frac12 \int |\nabla \tilde v|^2 + \frac14 \int (\tilde v)^4, $$
since it is conserved by the flow of cubic wave equation with no forcing. However, as we will see in Section 2, $\psi \notin L^2_\loc(\R^2)$, 
so we expect $\tilde v \notin H^1$, hence $E(\tilde v, \partial_t \tilde v) = +\infty$. 

In order to deal with the lack of regularity of $\tilde v$, we 
make use of the \linebreak I-method developed by Colliander, Keel, Staffilani, Takaoka and Tao in 2002 to show global well-posedness for dispersive equations with initial data of regularity below the energy space. We consider an operator $I = I_N$ given by the Fourier multiplier  $m_N$ with 
\begin{equation*}
m_N(\xi) = 
\begin{cases}
1 & \text{ for } |\xi| \le N,\\
\Big(\frac{N}\xi\Big)^{1-s} & \text{ for } |\xi| \ge 3N.
\end{cases}
\end{equation*}
For every $N$, one has that $I\tilde v \in H^1$ if and only if $\tilde v \in H^s$, and $I$ is the identity at low frequencies. Therefore, 
we can bound $\norm{\tilde v}_{H^s}$ by making use of the functional 
\begin{equation*}
F = E(I \tilde v(t), I \partial_t \tilde v(t)).
\end{equation*}
Of course, the functional $F$ will not be time independent, but we can bound it by making use of a Gronwall argument. We have that 
$$\frac \d{\d t} F = \text{ commutator terms } + \text{ forcing terms}$$ 
The commutator terms are typical of the I-method, and we can estimate them by
making use of simple harmonic analytic techniques. 

The terms coming from the forcing are more subtle to estimate, and in particular the term 
$$-3\int  Iv_t(Iv)^2I(\rho\psi) $$
requires a very sharp large deviation estimate for the $L^p$ norm of the term $I(\rho\psi)$. 

Putting the estimates together, for fixed $N$, allows us to show existence of $\tilde v$ for a random time $O(1)$. Iterating this procedure
on a sequence $N_0 \ll N_1 \ll N_2 \ll \dotsb$, we get global existence for $\tilde v$. 

We would like to point out that, as far as the author is aware, this is the first instance of an application of the $I$ - method that requires to change the cutoff parameter $N$ in a time-dependent way.

In section 5, we then conclude the proof of Theorem \ref{main} by showing that when $\rho \to 1$, the solutions $\tilde v$ converge to a limit which is continuous in the initial data $(u_0,u_1)$,
and that every solution $v$ of \eqref{SNLWv} must be equal to this limit.

\vspace{15pt}
\noindent
\textbf{Acknowledgements.} The author would like to thank is PhD supervisor Tadahiro Oh for suggesting the problem and his constant help and support. The author was supported by the European Research Council (grant no. 637995 ``ProbDynDispEq") and by The Maxwell Institute Graduate School in Analysis and its Applications, a Centre for Doctoral Training funded by the UK Engineering and Physical Sciences Research Council (grant EP/L016508/01), the Scottish Funding Council, Heriot-Watt University and the University of Edinburgh.
%\section{Notation}

\section{Stochastic Convolution}
In this section, we establish relevant estimates on the stochastic convolution 
$$\psi(t,\cdot) {:= \int_{0}^t \frac{\sin((t-t')|\nabla|)}{|\nabla|}} \d W_t.$$
Let $\psi_N(t,\cdot):= \int_0^t \frac{\sin((t-t')|\nabla|)}{|\nabla|}\chi_N(\nabla) \d W_t$, where 
$\chi_N$ is the indicator function of the ball of radius $N$ in $\R^2$. Following the construction in \cite{gko18}, 
we define $\wick{\psi^l_N}$ for $0 \le l \le 3$ in the following way 
\begin{align*}
\wick{\psi^0_N}(x,t) &:= 1, \\
\wick{\psi^1_N}(x,t) &:= \psi_N(x,t), \\
\wick{\psi^2_N}(x,t) &:= \psi^2_N(x,t) - \E[\psi^2_N(x,t)], \\
\wick{\psi^3_N}(x,t) &:= \psi_N^3(x,t) - 3 \E[\psi^2_N(x,t)] \psi_N(x,t)
\end{align*}
This corresponds with the Wick renormalisation obtained through the Hermite polynomial described in \cite{gko18}.
We will prove the following:
\begin{Proposition} \label{psi}
 Let $0\le l \le 3$, let $\rho: \R^2 \to \R$ a $C^\infty$ function with compact support, let $\epsilon > 0$, and let $T>0$.
 Let $1\le p,q,r<\infty$.
\begin{renumerate}
 \item For every $N$, $\psi_N \in C^\infty_{t,x}(\R\times \R^2)$ a.s..
 \item The sequence $\rho\wick{\psi_N^l}$ is almost surely (uniformly) in 
 $L^q_TW_x^{-\epsilon,r}$. More precisely, 
 $$\norm{\norm{\rho\wick{\psi_N^l}}_{L^q_TW_x^{-\epsilon,r}}}_{L^p(\Omega)} \le C_{T,p,q,\epsilon,r}. $$
 \item The sequence $\rho\wick{\psi_N^l}$ is Cauchy in $L^p_\omega L^q_TW_x^{-\epsilon,r}$. \label{Cauchy}
\end{renumerate}
Because of the arbitrariness of $\rho$, {\normalfont{(\ref{Cauchy})}} implies that $\wick{\psi_N^l}$ has a limit in 
$L^{p}(\Omega, L^q_T W^{-\epsilon,r}_{\loc})$. We call the limit of this sequence $\wick{\psi^l}$. We have that 
\begin{renumerate}
\setcounter{enumi}{3}
 \item $\norm{\rho \wick{\psi^l}}_{L^\infty_T W_x^{-\epsilon,r}} < \infty$ a.s., and $\rho \wick{\psi^l}(t)$ is
 a.s.\ continuous in $t$ with values in $W_x^{-\epsilon,r}$.
 \item Let $m: \R^+\cup\{0\}\to [0,1]$ be a smooth function such that $m(r) = 1$ for every $r\le 1$, and 
 $\partial_r^n m(r)\lesssim_n r^{-\epsilon - n}$ for $r\ge 1$, $n\ge 0$. Let $I_N$ be the operator associated to the Fourier multiplier 
 $m_N(|\xi|):=m(|\xi|/N)$, i.e.\ $\widehat{I_N\phi}(\xi) := m(|\xi|/N) \hat\phi(\xi)$. Then
$\P\left(\norm{I_N(\rho\psi)}_{L^p_{t,x}([0,T];\R^2)} > \lambda p^\frac12 \log^\frac12 N\right) \le C_{T,\rho,m}^p  \lambda^{-p}$.  
\end{renumerate}
 \end{Proposition}

In order to prove this proposition, we need some tools. Recall the Hermite polynomials generating function
\begin{equation}
 F(t,x;\sigma) := e^{tx-\frac12\sigma t^2} = \sum_{k=0}^\infty \frac{t^k}{k!} H_k(x;\sigma).
\end{equation}
For simplicity, let $F(t,x) := F(t,x;1)$, and let $H_k(x) := H_k(x;\sigma)$.
Fix $d\in\N$. Consider the Hilbert space $L^2(\R^d,\mu)$, where $\d\mu_d : = (2\pi)^{-\frac d2} \exp(-|x|^2/2).$
Then Hermite polynomials satisfy 
$$\int_\R H_k(x)H_m(x) \d \mu_1(x) = k! \delta_{km} $$
for all $k,m\in \N$. Define the homogeneous Wiener chaos of order $k$ to be an element of the form 
$\prod_{j=1}^d H_{k_j}(x_j)$, where $k = k_1+\dotsb+k_d$. Denote by $\H_k$ the closure of homogeneous Wiener
chaoses of order $k$ in $L^2(\R^d,\mu_d)$. Then, from the property $L^2(\R^d,\mu_d) = \bigotimes_{j=1}^d L^2(\R,\mu_1)$, 
we have the Ito-Wiener decomposition 
$$L^2(\R^d,\mu_d) = \bigoplus_{k=0}^\infty \H_k. $$
Consider the operator $L:= -(\Delta - x\cdot\nabla)$ (the Ornstein-Uhlenbeck operator). Then any element in $\H_k$
is an eigenvector of $L$ with eigenvalue $k$, so $\bigoplus_{k=0}^\infty \H_k$ is the spectral decomposition of
$L^2$ associated to $L$.\\
Moreover, we have the following hypercontractivity of the Ornstein-Uhlenbeck semigroup $U(t):= e^{-tL}$ due 
to Nelson \cite{n66}. 
\begin{Lemma}
 Let $q>1$ and $p\ge q$. Then, for every $u \in L^q(\R^d,\mu_d)$ and $t\ge \frac12\log(\frac{p-1}{q-1})$,
 we have 
 \begin{equation} \label{hyper}
\norm{U(t)u}_{L^p(\R^d,\mu_d)} \le \norm{u}_{L^q(\R^d,\mu_d)}.  
 \end{equation}

\end{Lemma}
Notice that the constant of the inequality in (\ref{hyper}) (i.e.\ $1$) and the range of $p,q,t$ do
not depend on the dimension $d$. As a consequence, the following holds.
\begin{Lemma}
 Let $F\in\H_k$. Then, for $p\ge2$, we have 
 \begin{equation}
  \norm{F}_{L^p(\R^d,\mu_d)} \le (p-1)^\frac k2 \norm{F}_{L^2(\R^d,\mu_d)}.
 \end{equation}
\end{Lemma}
This estimate follows simply by applying (\ref{hyper}) to $F$, setting $q=2$, $t= \frac12\log(p-1)$, and
recalling that $F$ is an eigenvector of $U(t)$ with eigenvalue $e^{-kt}$. As a further consequence,
we obtain the following lemma.
\begin{Lemma}\label{LEM:hyp3}
Fix $k \in \mathbb{N}$ and $c(n_1, \dots, n_k) \in \C$.
Given 	
 $d \in \mathbb{N}$, 
 let $\{ g_n\}_{n = 1}^d$ be 
 a sequence of  independent standard complex-valued Gaussian random variables.
Define $S_k(\omega)$ by 
\begin{equation*}
 S_k(\omega) = \sum_{\Gamma(k, d)} c(n_1, \dots, n_k) g_{n_1} (\omega)\cdots g_{n_k}(\omega),
%\label{hyp3}
\end{equation*}

\noindent
where $\Gamma(k, d)$ is defined by
\[ \Gamma(k, d) = \big\{ (n_1, \dots, n_k) \in \{1, \dots, d\}^k \big\}.\]
Then, for $p \geq 2$, we have
\begin{equation}
 \|S_k \|_{L^p(\Omega)} \leq \sqrt{k+1}(p-1)^\frac{k}{2}\|S_k\|_{L^2(\Omega)}.
\label{hyp4}
 \end{equation}
\end{Lemma}
%\todo{reference for the proof.}
\begin{proof}[Proof of Proposition \ref{psi}]
\hfill
\begin{renumerate}
 \item Consider the operator $A_N(t) := \frac{\sin(t|\nabla|)}{|\nabla|}\chi_N(\nabla)$.
 We have that $A_N(t) \phi = \phi \ast a_N(t,\cdot)$, where $a_N(t,x) = \int_{B_N} \frac{\sin(t|\xi|)}{|\xi|} 
 e^{2\pi i \xi \cdot x}\d \xi$. Moreover, 
 $$a_N \in H^{2m}([-T,T]\times\R^2)\text{ for every } m\in \N,$$ 
 since
 \begin{align*}
 \norm{a_N}_{\dot H^{2m}}^2 &\sim \norm{\partial_t^{2m} a_N}_{L^2}^2 + \norm{\Delta^m a_N}_{L^2}^2 \\
    & = 2\int_{-T}^T \int_{B_N} \sin(t|\xi|)^2 |\xi|^{4m-2} \d\xi \\
    & < +\infty.
 \end{align*}
 Therefore, for any ball $B\subseteq \R^2$,
 \begin{align*}
  \E\norm{\psi_N}_{\dot H^{2m}([-T,T] \times B)}^2\sim \E\norm{\partial_t^{2m} \psi_N}_{L^2([-T,T] \times B)}^2 + \E\norm{\Delta^m \psi_N}_{L^2([-T,T] \times B)}^2, \\  
 \end{align*}
 and
 \begin{align*}
  &\E \norm{\partial_t^{2m} \psi_N}_{L^2([-T,T] \times B)}^2 \\
  & = 
  \int_{[-T,T] \times B} \E \left|\int_{0}^{t}\int_{\R^2} \partial_t^{2m}a_N(t-t',x-x') \d W(t',x')\right|^2 \d t \d x\\
  & \le \int_{[-T,T] \times B} \norm{a_N}_{\dot H^{2m}([-T,T]\times \R^2)}^2 \d t\d x \\
  & = 2T|B| \norm{a_N}_{\dot H^{2m}([-T,T]\times \R^2)}^2 \\
  &< +\infty, \\
 \end{align*}
and similiarly, 
\begin{align*}
  &\E \norm{\Delta^{m} \psi_N}_{L^2([-T,T] \times B)}^2 \\
  & = 
  \int_{[-T,T] \times B} \E \left|\int_{0}^{t}\int_{\R^2} \Delta^{m}a_N(t-t',x-x') \d W(t',x')\right|^2 \d t \d x\\
  & \le \int_{[-T,T] \times B} \norm{a_N}_{\dot H^{2m}([-T,T]\times \R^2)}^2 \d t\d x \\
  & = 2T|B| \norm{a_N}_{\dot H^{2m}([-T,T]\times \R^2)}^2 \\
  &< +\infty, \\
 \end{align*}
so $\E\norm{\psi_N}_{\dot H^{2m}([-T,T] \times B)}^2 < +\infty$ for every $m$ a.s., therefore $\psi_N \in C^\infty_{t,x}$ a.s..
 \item We have that, for $s<t$, by Plancherel,\\
 \resizebox{\linewidth}{!}{
 \begin{minipage}{\linewidth}
 \begin{equation} \label{covariance}
   \begin{aligned}
  &\E[\psi_N(s,x)\psi_N(t,y)] \\
  &= \E \left[\int_0^s\int_{\R^2} a_N(s-t',x-x') \d W(t',x') \int_0^t\int_{\R^2} a_N(t-t',y-x') \d W(t',x')\right]\\
  &= \int_0^s\int_{\R^2} a_N(s-t',x-x')a_N(t-t',y-x') \d t' \d x'\\
  &= \int_0^s\int_{B_N} \frac{\sin((s-t')|\xi|)\sin((t-t')|\xi|)}{|\xi|^2} e^{-2\pi i \xi\cdot (x-y)} \d \xi \d t'.
 \end{aligned}
 \end{equation}
 \newline
 \end{minipage}
 }
Define $\gamma(t,\xi)$ by 
\begin{equation}
\gamma(t,\xi) := \int_0^t \frac{\sin((t-t')|\xi|)^2}{|\xi|^2} \d t'. 
\end{equation}
By applying the Bessel potentials $\jap{\nabla_x}^{-\epsilon}$ and $\jap{\nabla_y}^{-\epsilon}$ of order 
$\epsilon$ and setting $t=s$, $x=y$, we get 
\begin{align*}
 \E\left[\left|{\jap\nabla}^{-\epsilon} \psi_N(t,x)\right|^2\right] &= 
 \int_{|\xi|<N} {\jap\xi}^{-2\epsilon}\gamma(t,\xi) \d \xi \\
 &\lesssim t^3 + t\int_{|\xi|>1} \frac{1}{{\jap\xi}^{2+2\epsilon}} \\
 &\lesssim t^3 + t.
\end{align*}
for any $\epsilon >0$, $x\in \R^2$, and uniformly in $N$. In particular, by hypercontractivity (Lemma \ref{LEM:hyp3}),
we have that
\begin{equation*}
 \E\left[\left|{\jap\nabla}^{-\epsilon} \psi_N(t,x)\right|^p\right] \lesssim_{p,\epsilon,t} 1,
\end{equation*}
and thus, since $\rho\in C^\infty_c$, 
\begin{equation*}
 \E\left[\norm{\rho(\cdot)\psi_N(t,\cdot)}_{W^{-\epsilon,p}}^p\right] =  \E\left[\norm{{\jap\nabla}^{-\epsilon}\rho(\cdot)\psi_N(t,\cdot)}_{L^p}^p\right] < + \infty
\end{equation*}
for any $\epsilon>0$, $t>0$, and $p\ge1$, uniformly in $N\in \N$.\\
By the properties of Wick products \cite[Theorem I.22]{s74}, we have that
\begin{align*}
 \E\left[\wick{\psi_N^l(t,x)}\wick{\psi_N^l(t,y)}\right]
 & \sim \left\{\E\left[\psi_N(t,x)\psi_N(t,y)\right] \right\}^l \\
 &= \int_{|\xi_1|,\dotsc,|\xi_l|\le N} \prod_{j=1}^l \gamma(t,\xi_j) e^{-2\pi i \xi_j \cdot (x-y)} \d\xi_j \\
 & = \int_{|\xi_1|,\dotsc,|\xi_l|\le N} e^{-2\pi i \left(\sum_{1}^l\xi_j\right) \cdot (x-y)} \prod_{j=1}^l \gamma(t,\xi_j) \d\xi_j.
\end{align*}
Therefore, proceeding as before,
\begin{align*}
 \E\left[\left|{\jap\nabla}^{-\epsilon}\wick{\psi_N^l(t,\cdot)}(x)\right|^2\right] &= \int_{|\xi_1|,\dotsc,|\xi_l|\le N} \jap{\xi_1 + \dotsb + \xi_l}^{-2\epsilon} \prod_{j=1}^l \gamma(t,\xi_j) \d\xi_j \\
& \lesssim_t \int_{\xi_1,\dotsc,\xi_l} \jap{\xi_1 + \dotsb + \xi_l}^{-2\epsilon} \prod_{j=1}^l \jap{\xi_j}^{-2} \d\xi_j < \infty.
\end{align*}
for any $\epsilon >0$, $t>0$, uniformly in $x \in \R^2$ and $N$.
Hence we have 
\begin{equation*}
\left[\norm{\rho(\cdot)\wick{\psi_N(t,\cdot)^l}}_{H^{-\epsilon}}^2\right] \lesssim_{\rho,t} 1,
\end{equation*}
and by hypercontractivity
\begin{equation*}
\E\left[\norm{\rho(\cdot)\wick{\psi_N(t,\cdot)^l}}_{W^{-\epsilon,p}}^p\right]  \lesssim_{p,\rho,t} 1.
\end{equation*}
Integrating in time, we have that 
\begin{equation*}
\E\left[\norm{\rho(x)\wick{\psi_N(t,x)^l}}_{L^2_tH_x^{-\epsilon}([-T,T]\times \R^2)}^2\right] \lesssim_{\rho,T} 1,
\end{equation*}
and by hypercontractivity,
\begin{equation}\label{stlpwn}
\E\left[\norm{\rho(x)\wick{\psi_N(t,x)^l}}_{L^p_tW_x^{-\epsilon,p}([-T,T]\times \R^2)}^p\right] \lesssim_{p,\rho,T} 1.
\end{equation}
%Notice that, by hypercontractivity, the the implicit constant $C(p,\rho,T)$ can be chosen as 
%\begin{equation}\label{cstlpwn}
%C(p,\rho,T) = p^\frac{pl}2 \norm{\rho}_{W^{\epsilon,1}\cap W^{\epsilon,\infty}}^p C(T)^p,
%\end{equation}
%which holds even when $\rho$ is not compactly supported.\\
Moreover, since $\rho(\cdot)\wick{\psi_N^l}$ has compact support, $T< +\infty$, if
$q,r,s \le p$, it follows that 
\begin{equation*}
\norm{\norm{\rho(x)\wick{\psi_N(t,x)^l}}_{L^r_tW_x^{-\epsilon,s}([-T,T]\times \R^2)}}_{L^q(\Omega)} \lesssim_{p,\rho,T} 1.
\end{equation*}
\item Following (ii), we have that for $N\le M$:
\begin{align*}
&\E[\psi_N(s,x)\psi_M(t,y)] \\
&= \int_0^s\int_{B_N} \frac{\sin((s-t')|\xi|)\sin((t-t')|\xi|)}{|\xi|^2} e^{-2\pi i \xi\cdot (x-y)} \d \xi \d t'\\
&= \E[\psi_N(s,x)\psi_N(t,y)].
\end{align*}
Therefore, using the properties of Wick products \cite[Theorem I.22]{s74} again,
\begin{align*}
&\E\left|\wick{\psi_M^l(t,x)}-\wick{\psi_N^l(t,x)}\right|^2 \\
&= \E\left|\wick{\psi_M^l(t,x)}\right|^2 - 2 \E\left[\wick{\psi_M^l(t,x)}\wick{\psi_N^l(t,x)}\right] + \E\left|\wick{\psi_N^l(t,x)}\right|^2\\
&\sim \left\{\E\left[\psi_M(t,x)^2\right]\right\}^l- 2 \left\{\E\left[\psi_M(t,x)\psi_N(t,x)\right]\right\}^l + \left\{\E\left[\psi_N(t,x)^2\right]\right\}^l \\
&= \left\{\E\left[\psi_M(t,x)^2\right]\right\}^l - \left\{\E\left[\psi_N(t,x)^2\right]\right\}^l \\
& = \int_{N\le |\xi_1|,\dotsc,|\xi_l|\le M} e^{-2\pi i \left(\sum_{1}^l\xi_j\right) \cdot (x-y)} \prod_{j=1}^l \gamma(t,\xi_j) \d\xi_j.
\end{align*}
Similarly, we have
\begin{align*}
&\E\left|{\jap\nabla}^{-\epsilon} \left(\wick{\psi_M^l}-\wick{\psi_N^l}\right)(t,x)\right|^2\\
&= \int_{N\le |\xi_1|,\dotsc,|\xi_l|\le M} \jap{\xi_1 + \dotsb + \xi_l}^{-2\epsilon} \prod_{j=1}^l \gamma(t,\xi_j) \d\xi_j \\
& \lesssim_t \int_{|\xi_1|,\dotsc,|\xi_l| \ge N} \jap{\xi_1 + \dotsb + \xi_l}^{-2\epsilon} \prod_{j=1}^l \jap{\xi_j}^{-2} \d\xi_j \lesssim N^{-2\theta}.
\end{align*}
for every $0< \theta < \epsilon$. Integrating in space and time, we obtain 
\begin{equation*}
\E\left[\norm{\rho(x)\wick{\psi_M(t,x)^l} - \rho(x)\wick{\psi_N(t,x)^l}}_{L^2_tH_x^{-\epsilon}}^2\right]  \lesssim_{p,\rho,t} N^{-2 \theta},
\end{equation*}
from which, by hypercontractivity
\begin{equation*}
\E\left[\norm{\rho(x)\wick{\psi_M(t,x)^l} - \rho(x)\wick{\psi_N(t,x)^l}}_{W^{-\epsilon,p}}^p\right]  \lesssim_{p,\rho,t} N^{-p \theta},
\end{equation*}
and, arguing as in (ii), 
\begin{equation*}
\norm{\norm{\rho(x)\wick{\psi_M(t,x)^l}-\rho(x)\wick{\psi_N(t,x)^l}}_{L^p_tW_x^{-\epsilon,q}([-T,T]\times \R^2)}}_{L^q(\Omega)}
\lesssim_{\max(p,q,r),\rho,T} N^{-\theta}.
\end{equation*}
\item Using the formula (\ref{covariance}), we have that 
\begin{align*}
&\E\left[\wick{\psi_N(s,x)^l} \wick{\psi_N(t,y)^l}\right] \\
& \sim \E[\psi_N(s,x)\psi_N(t,y)]^l\\ 
&= \left(\int_0^{s\wedge t}\int_{B_N} \frac{\sin((s-t')|\xi|)\sin((t-t')|\xi|)}{|\xi|^2} e^{-2\pi i \xi\cdot (x-y)} \d \xi \d t'\right)^l.
\end{align*}
Therefore, calling 
$$\gamma_N(s,t;\xi):= \int_0^{s\wedge t} \frac{\sin((s-t')|\xi|)\sin((t-t')|\xi|)}{|\xi|^2} \d t' \lesssim_{s,t} \jap{\xi}^{-2}$$
and proceeding as in (iii), we have that
\begin{align*}
&\E\left[\left(\jap{\nabla}^{-\epsilon}\wick{\psi_N(s,\cdot)^l}(x)\right)\left( \jap{\nabla}^{-\epsilon} \wick{\psi_N(t,\cdot)^l}(x)\right)\right] \\
& \sim \int_{|\xi_1|,\dotsc,|\xi_l|\le N} \jap{\xi_1+\dotsb+\xi_l}^{-2\epsilon} \prod_{j=1}^l \gamma(s,t;\xi_j) \d\xi_j.
\end{align*}
Therefore, 
\begin{align*}
&\E\left|\left(\jap{\nabla}^{-\epsilon}\wick{\psi_N(t+h,\cdot)^l}(x)\right) - \left( \jap{\nabla}^{-\epsilon} \wick{\psi_N(t,\cdot)^l}(x)\right)\right|^2 \\
& \begin{multlined} 
\sim\int_{|\xi_1|,\dotsc,|\xi_l|\le N} \jap{\xi_1+\dotsb+\xi_l}^{-2\epsilon} \\ \times  
\left(\prod_{j=1}^l\gamma(t+h,t+h;\xi_j) -2 \prod_{j=1}^l\gamma(t,t+h;\xi_j) + \prod_{j=1}^l\gamma(t,t;\xi_j)\right) \d\xi_j.
\end{multlined}.\\
\end{align*}
Interpolating between 
$$\gamma(t+h,t+h;\xi) - \gamma(t,t+h;\xi), \gamma(t,t+h;\xi) - \gamma(t,t;\xi) \lesssim_t |h|\jap{\xi}^{-1}$$
and 
$$\gamma(t+h,t+h;\xi) - \gamma(t,t+h;\xi), \gamma(t,t+h;\xi) - \gamma(t,t;\xi) \lesssim_t \jap{\xi}^{-2},$$
we get 
$$\gamma(t+h,t+h;\xi) - \gamma(t,t+h;\xi), \gamma(t,t+h;\xi) - \gamma(t,t;\xi) \lesssim_t |h|^\theta \jap{\xi}^{-2+\theta}$$
for every $0\le \theta \le 1$, so choosing $\theta < 2\epsilon$, by the discrete Liebnitz formula,
\begin{align*}
&\E\left|(\jap{\nabla}^{-\epsilon}\wick{\psi_N(t+h,\cdot)^l})(x) - \left( \jap{\nabla}^{-\epsilon} \wick{\psi_N(t,\cdot)^l}(x)\right)\right|^2 \\
& \lesssim_{t,l} \int_{\xi_1,\dotsc,\xi_l} \jap{\xi_1+\dotsb+\xi_l}^{-2\epsilon} |h|^\theta \jap{\xi_1}^{-2+\theta} \prod_{j=2}^l \jap{\xi_j}^{-2}
 \d\xi_j \lesssim_{t,l} |h|^{\theta}.
\end{align*}
Since this inequality passes to the limit, we have that
\begin{align*}
&\E\left|\left(\jap{\nabla}^{-\epsilon}\wick{\psi(t+h,\cdot)^l}(x)\right) - \left( \jap{\nabla}^{-\epsilon} \wick{\psi(t,\cdot)^l}(x)\right)\right|^2 \lesssim_t |h|^{\theta}
\end{align*} 
for a.e. $\!t$. Integrating and using hypercontractivity, we get 
\begin{align*}
&\E\left[\norm{\rho(\cdot)\left(\wick{\psi(t+h,\cdot)^l}- \wick{\psi(t,\cdot)^l}\right)}_{W^{-\epsilon,p}}^p\right] \lesssim_{t,p} |h|^{\frac{p\theta}{2}}.
\end{align*} 
Moreover, if $|t|\le T$, the implicit constant $C$ can be chosen as $C=C(T,p)$. Therefore, by Kolmogorov Continuity Theorem, for $|t|<T$, 
we have that $\rho\wick{\psi^l} \in C^\alpha W^{-\epsilon,p}$, for every 
$\alpha < \frac{p\theta}{2(p-1)}$, so $\rho\wick{\psi^l} \in C_tW^{-\epsilon,p}$ a.s..
\item In order to prove this, we just need to prove
\begin{equation*}
\E\left[\norm{I_N(\rho \psi)}_{L^p}^p\right] \le C(T,\rho,m)^p p^\frac p2 \log^\frac p 2 N ,
\end{equation*}
then (v) will follow from a straightforward application of Chebishev inequality. \\
Proceeding as in \eqref{covariance}, we have that 
\begin{align*}
&\phantom{=} \E[|I_N(\rho\psi)|^2(t,x)] \\
&= \E\Big[ \Big|\iint m_N(\xi + \eta) \hat\psi(\xi) \hat\rho(\eta) e^{i2\pi(\xi+\eta)\cdot x} \d \xi \d \eta\Big|^2\Big]\\
&= \iiint m_N(\xi + \eta_1)\overline{m_N(\xi + \eta_2)} \gamma(t, \xi) \hat\rho(\eta_1)\overline{\hat\rho(\eta_2)} e^{i2\pi(\eta_1-\eta_2)\cdot x} \d \xi \d \eta_1 \d \eta_2,
\end{align*}
which is the inverse Fourier transform of the function $F:\R^2 \times \R^2 \to \C$,
$$F(\eta_1,\eta_2) = \hat\rho(\eta_1) \overline{\hat \rho(\eta_2)} \int m_N(\xi + \eta_1)\overline{m_N(\xi + \eta_2)} \gamma(t, \xi) \d \xi,$$
restricted to the plane $\{(x,-x)\}$. We have that $\gamma(t,\xi) \lesssim_t \jap{\xi}^{-2}$, and 
that for $n \ge 1$, $\partial_\eta^n m_N(\xi + \eta) \lesssim_{m,n} \jap{\xi + \eta}^{-n}$, therefore calling 
$$\phi(\eta_1,\eta_2) := \int m_N(\xi + \eta_1)\overline{m_N(\xi + \eta_2)} \gamma(t, \xi) \d \xi,$$
we have that 
$$\norm{\nabla^n_{\eta_1,\eta_2}\phi}_{L^\infty} \lesssim_{n,m} 
\left\{\begin{array}{ll}
\log N & \text{if } n=0 \\
1 & \text{if } n\ge 1
\end{array}\right.,
 $$
hence $\norm{\phi}_{C^k} \lesssim_{t,k} \log N$. Since $\hat\rho(\eta_1) \overline{\hat \rho(\eta_2)}$ is a Schwartz function of $\eta_1,\eta_2$, this implies that 
\begin{equation}
\E[|I_N(\rho\psi)|^2(t,x)] \lesssim_{T,\rho,m} \frac{\log N}{\jap{x}^4}.
\end{equation}
By hypercontractivity, since $I_N(\rho\psi)$ is gaussian, one gets that 
\begin{equation}
\E[|I_N(\rho\psi)|^p(t,x)] \le C(T,\rho,m)^p p^\frac p2 \frac{\log^\frac p2 N}{\jap{x}^{4p}}, 
\end{equation}
hence, integrating, 
$$\E\left[\norm{I_N(\rho \psi)}_{L^p_{t,x}}^p\right] \le C(T,\rho,m)^p p^\frac p2 \log^\frac p 2 N. $$
\end{renumerate}
\end{proof}
\section{Local well-posedness for the localized equation}
Take a compactly supported $\rho \in C^\infty_c(\R^2)$, and consider the equation 
\begin{equation}\label{classicalLSNLW}
\begin{gathered}
v_{tt}  - \Delta v + v^3 + 3v^2 \rho\psi + 3v \rho \wick{\psi^2} + \rho \wick{\psi^3} = 0, \\
v(t_0,x) = u_0(x) \hspace{1 cm}
\partial_t v(t_0,x) = u_1(x).
\end{gathered}
\end{equation}
Notice that (at least formally) 
$$v^3 + 3v^2 \rho\psi + 3v \rho \wick{\psi^2} + \rho \wick{\psi^3} = \wick{(v+\psi)^3}$$
whenever $\rho =1$. \\
Consider as well the mild formulation of (\ref{classicalLSNLW})
\begin{multline}\label{LSNLW} \tag{LSNLW}
v(t) = \cos((t-t_0)|\nabla|) u_0 + \frac{\sin((t-t_0)|\nabla|)}{|\nabla|} u_1 + 
\int_{t_0}^t \frac{\sin((t-t')|\nabla|)}{|\nabla|} \\
\times \left(v^3 + \sum_{j=0}^2 \binom{3}{j} v^j\rho\wick{\psi^{3-j}}\right).
\end{multline}
The following local well-posedness result holds:
\begin{Proposition}[Local Well-Posedness] \label{LWP}
Let $1> s \ge \frac 23$, and let $(u_0,u_1) \in \H^s$. Than there exists a time 
$\tau = \tau(t_0,\omega, \norm{(u_0,u_1)}_{\H^s}) > 0$ a.s., nonincreasing in $\norm{(u_0,u_1)}_{\H^s}$, such that the equation \eqref{LSNLW} has a unique solution
in the space $C([t_0 - \tau, t_0 + \tau];H^s)$. Moreover, a.s.\ in $\omega$, we have that $\inf_{|t_0| < T} \tau(t_0, \omega, \norm{(u_0,u_1)}_{\H^s}) > 0$
for every $T < \infty$.
\end{Proposition}
\begin{Remark}
For the sake of simplicity, in this work we will use just Sobolev embeddings to get the required estimates, obtaing the local well-posedness result
in the space $H^s$ for $s \ge \frac23$. Since the constraint coming from Section 5 is stronger than this one, it is enough for our purposes. However, improving the proof by making use of the Strichartz estimates for the wave equation, it is possible to relax the condition to
$s> \frac 14$. See \cite{gko18} for this analysis, which can be applied with very little modifications to this problem. 
\end{Remark}
\begin{proof}
This proposition will follow by a standard fixed point argument. Consider the map 
\begin{equation}
\begin{aligned}
\Gamma_{(u_0,u_1)} v &:= \cos((t-t_0)|\nabla|) u_0 + \frac{\sin((t-t_0)|\nabla|)}{|\nabla|} u_1 \\
&+ \int_{t_0}^t \frac{\sin((t-t')|\nabla|)}{|\nabla|} \left(v^3 + \sum_{j=0}^2 \binom{3}{j} v^j\rho\wick{\psi^{3-j}}\right).
\end{aligned}
\end{equation}
defined on functions $v\in C_t([t_0-\tau,t_0+\tau];H^s)$.
By the Sobolev embedding $H^s \to H^{\frac23} \to W^{0+,6-} \to L^6$, and using that $s-1 < 0$ and the fact that $\rho\wick{\psi^k}$ is compactly supported, we have that 
\begin{align*}
\norm{\Gamma_{(u_0,u_1)} v}_{H^s} 
&\le \norm{(u_0,u_1)}_{\H^s} +  2\tau\norm{v^3}_{L^{\infty}_tH^{s-1}_x} + \sum_{j=0}^2 \binom3j \norm{v^j\rho\wick{\psi^{3-j}}}_{L^{\infty}_tH^{s-1}_x} \\
& \lesssim \norm{(u_0,u_1)}_{\H^s} + 2\tau \norm{v}_{L^{\infty}_tL^6}^3 + 
\sum_{j=0}^2 \binom3j \norm{v}_{L^\infty_tW^{0+,6-}}^j \norm{\rho\wick{\psi^{3-j}}}_{L^1_tW^{0-,6+}_x}  \\
& \lesssim \norm{(u_0,u_1)}_{\H^s} +  2\tau\norm{v}_{L^{\infty}_tH^s_x}^3 + 
\sum_{j=0}^2 \binom3j \norm{v}_{L^\infty_tH^s_x}^j \norm{\rho\wick{\psi^{3-j}}}_{L^1_tW^{0-,6+}_x} \\
&\lesssim \norm{(u_0,u_1)}_{\H^s} + \tau \norm{v}_{C_tH^s}^3 + 
(1+\norm{v}_{C_tH^s}^2) \max_{1\le j \le 3} \norm{\rho\wick{\psi^j}}_{L^1_tW^{0-,6+}_x}.
\end{align*}
By Proposition \ref{psi}, $\norm{\rho\wick{\psi^j}}_{L^1_tW^{0-,6+}_x}$ is finite a.s. (locally in time), so it is
possibile to find $\tau \ll 1$ such that $\max_{1\le j \le 3} \norm{\rho\wick{\psi^j}}_{L^1_tW^{0-,6+}_x} \ll 1$. 
Therefore, for $\tau$ small enough, we have that if $\norm{v}_{C_tH^s} \le 2\norm{(u_0,u_1)}_{\H^s}$, 
then $\norm{\Gamma_{u_0,u_1}v}_{C_tH^s} \le 2 \norm{(u_0,u_1)}_{\H^s}$.\\
Proceeding similarly, we have that 
\begin{align*}
\norm{\Gamma_{u_0,u_1} v-\Gamma_{u_0,u_1} w} 
&\lesssim \tau \norm{v-w}_{C_tH^s}(\norm{v}^2_{C_tH^s} + \norm{w}^2_{C_tH^s}) \\
&+ 
\norm{v-w}_{C_t\H^s}(1+\norm{v}_{C_t\H^s} +\norm{w}_{C_t\H^s}) \max_{j=1,2} \norm{\rho \wick{\psi^j}}_{L^1_tW^{0-,6+}_x}.
\end{align*}
for the same reason, choosing $\tau \ll 1$ we have that $\Gamma_{(u_0,u_1)}$ is a contraption on $B_{2\norm{(u_0,u_1)}_{\H^s}} \subseteq C_t\H^s$.\\
In order to finish the proof of this proposition, we just need to show that 
$$\inf_{|t_0| < T} \tau(t_0, \omega, \norm{(u_0,u_1)}_{\H^s}) > 0$$
for every $T > 0$. Notice that, by the previous argument, in order to have that 
$\Gamma_{u_0,u_1}$ is a contraption, we just need that $\tau, \norm{\rho\wick{\psi^j}}_{L^1_t\H^{s-1}_x} < \delta$
for a certain fixed $\delta$. By Proposition \ref{psi}, we have that 
$\norm{\rho\wick{\psi^j}}_{L^2_tW^{0-,6+}_x([-T-1,T+1]\times \R^2)} < +\infty$ a.s.. Therefore, we have that (when $\tau < 1$),
\begin{equation*}
\norm{\rho\wick{\psi^j}}_{L^1_tW^{0-,6-}_x([t_0-\tau,t_0+\tau]\times \R^2)} \lesssim \tau^\frac 12
\norm{\rho\wick{\psi^j}}_{L^2_tW^{0-,6+}_x([-T-1,T+1]\times \R^2)},
\end{equation*}
therefore for fixed $T$, $\tau$ can be chosen independently from $t_0$, and we have \linebreak $\inf_{|t_0| < T} \tau(t_0, \omega, \norm{(u_0,u_1)}_{\H^s}) > 0$.
\end{proof}
\begin{Proposition}[Blow up condition]\label{blowup}
Let $T^*$ be the maximal time for which the solution to \eqref{LSNLW} exists in the interval $[0,T^*)$, in the sense that $(v,v_t) \in C([0,T^*);\H^s)$ and for every $\delta >0$, there is no function $(\tilde v, \tilde v_t)$ such that $(v,v_t) \in C([0,T^*+\delta);H^s)$ which solves \eqref{LSNLW}. Suppose that $T^* < +\infty$. Then we have 
\begin{equation} \label{eq:blowup}
\lim_{s\to T^*} \norm{(v(s),v_t(s))}_{\H^s} = + \infty.
\end{equation}
\end{Proposition}
\begin{proof}
Suppose by contradiction that $\lim_{s\to T^*} \norm{v(s),v_t(s)}_{\H^s} < +\infty$. Then by continuity in time,
$$\sup_{s < T^*} \norm{v(s),v_t(s)}_{\H^s} = M < +\infty.$$
Let $\tau' = \inf_{|t_0|< T^*} \tau(t_0,\omega,M),$ which is positive by Proposition \ref{LWP}. Consider the equation \eqref{LSNLW} starting from time $t_0 = T^* - \frac{\tau'}{2}$, with initial data $(v(t_0),v_t(t_0))$. By Proposition \ref{LWP} again, this will have a unique solution 
$$(\bar v, \bar v_t) \in C([t_0 - \tau', t_0 + \tau'];\H^s) = C\Big(\Big[T^* - \frac32\tau', T^* + \frac12 \tau'\Big];\H^s\Big),$$
which by uniqueness will coincide with $v$ in the interval $[T^* - \frac32\tau', T^*)$.
Therefore, defining
\begin{equation*}
(\tilde v, \tilde v_t)(s) = \left\{
\begin{aligned}
(v,v_t)(s) && \text{ for } 0\le s<T^*, \\
(\bar v,\bar v_t)(s) && \text{ for } T^* \le s\le T^* + \tau', \\
\end{aligned}\right.
\end{equation*}
$(\tilde v, \tilde v_t)$ will satisfy \eqref{LSNLW} in the interval $[0,T^*+\tau']$, contradiction.
\end{proof}

\section{Global well-posedness for the localized equation}
In this section, we establish global well-posedness for the equation \eqref{LSNLW}. In particular, we will prove the following
\begin{Proposition} \label{lgwp}
 Let $s > \frac 45$. Then the solution to \eqref{SNLW} with $(u_0,u_1) \in \H^s$ can be extended a.s.\ to a global solution 
 $u: \R \times \T^2 \to \R$. \\
 More precisely, for every $T>0$, $\delta > 0$, there exists a set $\Omega_{T,\delta}$ such that
 \begin{itemize}
  \item $\P((\Omega_{T,\epsilon})^c) \le \delta$, \\
  \item For every $\omega \in \Omega_{T,\delta}$, there exists a unique solution $u:[-T, T] \times \R^2 \to \R$ to \eqref{SNLW} such that
  $v(0,x) = u_0(x), v_t(0,x) = u_1(x)$. Moreover, this solution satisfies the estimate 
  $\norm{v}_{L^\infty_t([-T,T];H^s_x)} \le C(T, s, \delta). $
 \end{itemize}
\end{Proposition}
Let $m:\R^2\to\R$ be a smooth radial function, $0\le m \le 1$, such that 
\begin{equation}
m(\xi) =
\left\{\begin{aligned}
1 &&\text{ for } |\xi| \le 1,\\
\frac1{|\xi|^\epsilon} && \text{ for } |\xi| \ge 3.
\end{aligned}
\right.
\end{equation}
Let $m_N(\xi):= m(\frac\xi N)$. This way, $m_N$ will satisfy 
\begin{equation}
m_N(\xi) =
\left\{\begin{aligned}
1 &&\text{ for } |\xi| \le N,\\
{\Big(\frac N{|\xi|}\Big)}^\epsilon && \text{ for } |\xi| \ge 3N.
\end{aligned}
\right.
\end{equation}
Let $I_N$ the operator corresponding to the Fourier multiplier $m_N$, i.e.\ $\widehat{I_Nf}(\xi) = m_N(\xi)\widehat f(\xi)$. By the definition of the multiplier, for every $\sigma\in\R$, $1< p< +\infty$, $0\le\delta\le\epsilon$, $I_N$ will satisfy 
\begin{gather}
\norm{I_Nf}_{W^{\sigma+ \delta, p}} \lesssim N^\delta \norm{f}_{W^{\sigma,p}}, \label{I} \\
\norm{f}_{W^{\sigma,p}} \lesssim \norm{I_Nf}_{W^{\sigma+\epsilon,p}}. \label{inverseI}
\end{gather}
One can establish these estimates by showing the analogous ones for $I_1$, and then the $N$-dependence will follow from a simple scaling argument. \\
From \eqref{I},\eqref{inverseI}, one has that, for fixed $N$, $\norm{\cdot}_{\H^s}$ is equivalent to $\norm{I_N\cdot}_{\H^{s+\epsilon}}$. Therefore, by the blowup condition \eqref{eq:blowup}, in order to show existence of solutions up to time $T$, it is enough to show that 
\begin{equation} \sup_{|s|<T} \norm{(I_Nv,I_Nv_t)}_{\H^{s+\epsilon}} < +\infty. \label{eq:blowup2}\end{equation}
By taking $\epsilon = 1-s$, 
\begin{equation}
E(v,v_t) := \frac12 \int v_t^2 + \frac12\int v^2 + \frac12 \int |\nabla v|^2 + \frac14 \int v^4,
\end{equation}
we clearly have that 
$$E(I_Nv,I_Nv_t) \gtrsim \norm{(I_Nv,I_Nv_t)}_{\H^1}^2 \gtrsim \norm{v,v_t}_{\H^s}^2.$$
The goal of this section will henceforth beshowing finiteness of $E(I_Nv,I_Nv_t)$. In the following, we will abuse of notation and omit the subscript $N$ whenever it is not important in the analysis, writing $I$ instead of $I_N$. Similarly, we will write $E$ instead of $E(I_Nv,I_Nv_t)$, and $E(s)$ instead of $E(I_Nv(s),I_Nv_t(s))$.
\begin{Lemma}\label{timeDerivative}
\begin{align}
\phantom{=}& \frac \d{\d t} E(t) \notag \\
=&~-3\int  Iv_t(Iv)^2I(\rho\psi) \label{worst} \\
-&~3\int Iv_tIvI(\rho\wick{\psi^2}) - \int Iv_tI(\rho\wick{\psi^3}) \label{tame}\\
+&~\begin{multlined}
\int Iv_t\big[\big((Iv)^3-I(v^3)\big) + 3\big((Iv)^2I(\rho\psi) - I(v^2\rho\psi)\big) \\
+3\big((Iv)I(\rho\wick{\psi^2}) - I(v\rho\wick{\psi^2})\big)\big]
\end{multlined}\label{commutators}\\
+&\int Iv_tIv \label{harmless}.
\end{align}
\end{Lemma}
\begin{proof}
We will show this proposition by a formal computation using \eqref{classicalLSNLW}. This computation can be made rigorous a posteriori by using the estimates of this section. We omit this part of the argument.\\
By \eqref{classicalLSNLW}, 
\begin{align*}
\phantom{=}&\frac \d{\d t} E(t) \\
=& \int Iv_{t}\Big(Iv_{tt} - I\Delta v + (Iv)^3 + Iv\Big)\\
=& \int Iv_t\Big(- I\big(v^3 + 3v^2 \rho\psi + 3v \rho \wick{\psi^2} + \rho \wick{\psi^3}\big)+ (Iv)^3 +Iv\Big).
\end{align*}
The lemma follows by adding and subtracting the terms $3(Iv)^2 I(\rho\psi)$ and $3(Iv) I(\rho \wick{\psi^2}~)$.
\end{proof}
We will now proceed to estimate the various terms of the time derivative of $E(Iv,Iv_t)$, with the goal of applying a Gronwall argument. The terms \eqref{harmless},\eqref{tame} are relatively harmless. Estimating the commutator terms in \eqref{commutators} is the core of the I-method, and will take most of this section. 
However, from a technical point of view, the hardest term to estimate will be \eqref{worst}, which will also give the main contribution to the estimate on the growth of $E$. This term is also what makes the iteration of the I-method with varying $N$ necessary. 
\begin{Lemma}
\begin{equation}\label{harmlessestimate}
\eqref{harmless} \lesssim E(Iv,Iv_t).
\end{equation}
\end{Lemma}
\begin{proof}
By H\"older, 
$$\eqref{harmless} \le \norm{Iv_t}_{L^2}\norm{Iv}_{L^2} \le E(Iv,Iv_t). $$
\end{proof}
\begin{Lemma} 
For every $\gamma > 0$,
\begin{equation} \label{tameestimate}
\eqref{tame} \lesssim N^\gamma \Big(E(Iv,Iv_t)^\frac34 \norm{\rho\wick{\psi^2}}_{W^{-\gamma,4}} + E(Iv,Iv_t)^\frac12 \norm{\rho\wick{\psi^3}}_{H^{-\gamma}} \Big).
\end{equation}
\end{Lemma}
\begin{proof}
By H\"older and \eqref{I},
\begin{align*}
&~-3\int Iv_tIvI(\rho\wick{\psi^2}) - \int Iv_tI(\rho\wick{\psi^3})\\
\lesssim&~ \norm{Iv_t}_{L^2}\norm{Iv}_{L^4}\norm{I(\rho\wick{\psi^2})}_{L^4} + \norm{v_t}_{L^2}\norm{I(\rho\wick{\psi^3})}_{L^2}\\
\lesssim&~N^\gamma \Big(E(Iv,Iv_t)^{\frac12+\frac14} \norm{\rho\wick{\psi^2}}_{W^{-\gamma,4}} + E(Iv,Iv_t)^\frac12 \norm{\rho\wick{\psi^3}}_{H^{-\gamma}} \Big).
\end{align*}
\end{proof}
\begin{Lemma}\label{commI}
Let $k\le 3$. Then 
\begin{equation}\label{eq:commI}
\norm{(Iv)^k - I(v^k)}_{L^2} \lesssim_s N^{-(1-k(1-s))} \norm{Iv}^k_{H^1}. 
\end{equation}
\end{Lemma}
\begin{proof}
Let $v_{\lesssim N} = \int_{|\xi|< N/3} \hat v(\xi) e^{i\xi\cdot x}$, and let $v_{\gtrsim N} := v-v_{\lesssim N}$.\\
Since $\widehat {v_{\lesssim N}} (\xi) \neq 0$ only if $|\xi|<N/3$, by definition of $I$ we have that $I v_{\lesssim N} = v_{\lesssim N}$.
Similarly, $\widehat{v_{\lesssim N}^k} (\xi) \neq 0 $ only if $|\xi|< kN/3$, so $I\big(v_{\lesssim N}^k\big) = v_{\lesssim N}^k$.
Therefore, 
\begin{align*}
& (Iv)^k - I(v^k) \\
= &  \big(I(v_{\lesssim N} + v_{\gtrsim N})\big)^k - I\big((v_{\lesssim N} + v_{\gtrsim N})^k\big) \\
= & \big(v_{\lesssim N} + I(v_{\gtrsim N})\big)^k - I\big((v_{\lesssim N} + v_{\gtrsim N})^k\big)\\
= & v_{\lesssim N}^k - I\big(v_{\lesssim N}^k\big) + \sum_{l=0}^{k-1} \binom{k}{l}\left((Iv_{\gtrsim N})v_{\lesssim N}^l(Iv_{\gtrsim N})^{k-l-1} - I\big(v_{\gtrsim N}v_{\lesssim N}^lv_{\gtrsim N}^{k-l-1}\big)\right)\\
= & \sum_{l=0}^{k-1} \binom{k}{l}\left((Iv_{\gtrsim N})v_{\lesssim N}^l(Iv_{\gtrsim N})^{k-l-1} - I\big(v_{\gtrsim N}v_{\lesssim N}^lv_{\gtrsim N}^{k-l-1}\big)\right)
\end{align*}
Therefore, \eqref{eq:commI} follows if we prove that for every $l\le k-1$, 
\begin{equation}
\norm{(Iv_{\gtrsim N})v_{\lesssim N}^l(Iv_{\gtrsim N})^{k-l-1}}_{L^2}\lesssim N^{1-k(1-s)} \norm{Iv}_{H^1}^k
\end{equation}
and 
\begin{equation}
\norm{I\big(v_{\gtrsim N}v_{\lesssim N}^lv_{\gtrsim N}^{k-l-1}\big)}_{L^2}\lesssim N^{1-k(1-s)} \norm{Iv}_{H^1}^k.
\end{equation}
Let $\epsilon:= 1-s$. By Sobolev embeddings, we have that $\norm{f}_{L^{\frac2\epsilon}} \lesssim \norm{f}_{H^{1-\epsilon}}$. By H\"older,
we have $$\norm{f}_{L^{\left(\frac12 - \frac{(k-1)\epsilon}{2}\right)^{-1}}} \lesssim \norm{f}_{L^2}^{\frac{1-k\epsilon}{1-\epsilon}}\norm{f}_{L^\frac2\epsilon}^{\frac{(k-1)\epsilon}{1-\epsilon}}
\lesssim \norm{f}_{L^2}^{\frac{1-k\epsilon}{1-\epsilon}} \norm{f}_{H^{1-\epsilon}}^{\frac{(k-1)\epsilon}{1-\epsilon}}.$$ Therefore, again by H\"older,
we have:
\begin{align*}
 &~\norm{(Iv_{\gtrsim N})v_{\lesssim N}^l(Iv_{\gtrsim N})^{k-l-1}}_{L^2}\\
\lesssim &~\norm{(Iv_{\gtrsim N})}_{L^{\left(\frac12 - \frac{(k-1)\epsilon}{2}\right)^{-1}}} 
 \norm{Iv_{\gtrsim N}}_{L^{\frac2\epsilon}}^{l}\norm{Iv_{\gtrsim N}}_{L^{\frac2\epsilon}}^{k-l-1} \\
\lesssim &~\norm{Iv_{\gtrsim N}}_{L^2}^{\frac{1-k\epsilon}{1-\epsilon}}\norm{Iv_{\gtrsim N}}^{\frac{(k-1)\epsilon}{1-\epsilon}}_{H^{1-\epsilon}}\norm{Iv}_{H^{1-\epsilon}}^{k-1} \\
\lesssim &~N^{-(1-\epsilon)\frac{1-k\epsilon}{1-\epsilon}} \norm{Iv}_{H^{1-\epsilon}}^{\frac{1-k\epsilon}{1-\epsilon}}\norm{Iv_{\gtrsim N}}^{\frac{(k-1)\epsilon}{1-\epsilon}}_{H^{1-\epsilon}}\norm{Iv}_{H^{1-\epsilon}}^{k-1} \\
\lesssim &~N^{-(1-k\epsilon)} \norm{Iv}_{H^{1-\epsilon}}^k.
\end{align*}
Proceeding similarly and using \eqref{inverseI}, we have
\begin{align*}
 \norm{I\big(v_{\gtrsim N}v_{\lesssim N}^lv_{\gtrsim N}^{k-l-1}\big)}_{L^2} & \lesssim \norm{\big(v_{\gtrsim N}v_{\lesssim N}^lv_{\gtrsim N}^{k-l-1}\big)}_{L^2} \\
 & \lesssim N^{-(1-k\epsilon)} \norm{v}_{H^{1-\epsilon}}^k \\
 & \lesssim N^{-(1-k\epsilon)} \norm{Iv}_{H^{1}}^k.
\end{align*}
\end{proof}
\begin{Lemma} \label{commII}
 For every $\gamma > 0$, $0<\tilde s < 1$, there exist $p(\gamma) > 1, \eta(\gamma)> 0$ such that  
 \begin{equation} \label{eq:commII}
 \norm{(If)(Ig) - I(fg)}_{L^2} \lesssim_{\gamma, \tilde s} N^{\gamma - \frac{1-\tilde s}{2}} \norm{f}_{H^{1-\tilde s}}\norm{g}_{W^{-\eta(\gamma),1}\cap W^{-\eta(\gamma),p(\gamma)}} 
 \end{equation}
\end{Lemma}
\begin{proof}
 As in the proof of Lemma \ref{commI}, let us define 
 $$ u_{\lesssim M} := \int_{|\xi|< M/3} \widehat u(\xi) e^{i\xi\cdot x}$$
 and $u_{\gtrsim M} := u - u_{\lesssim M}$. Writing $f = f_{\lesssim N^{\frac 12}} + f_{\gtrsim N^{\frac12}}$ 
 and $g= g_{\lesssim N} + g_{\gtrsim N}$, we have that
 \begin{align}
&~ (If)(Ig) - I(fg) \notag\\
 =&~\left(If_{\lesssim N^{\frac12}}\middle)\middle(Ig_{\lesssim N}\right) - I\left(f_{\lesssim N^{\frac12}}g_{\lesssim N}\right)\tag{I}\\
+&~\left(If_{\lesssim N^{\frac12}}\middle)\middle(Ig_{\gtrsim N}\right) - I\left(f_{\lesssim N^{\frac12}}g_{\gtrsim N}\right)\tag{II}\\
+&~\left(If_{\gtrsim N^{\frac12}}\middle)\middle(Ig\right)\tag{III}\\
-&~I\left(f_{\gtrsim N^{\frac12}}g\right)\tag{IV}
 \end{align}
We have that
\begin{itemize}
 \item $\mathrm I=0$, since $If_{\lesssim N^{\frac12}} = f_{\lesssim N^{\frac12}}$, $Ig_{\lesssim N} = g_{\lesssim N}$ by
 definition of $I$, and $(f_{\lesssim N^{\frac12}}g_{\lesssim N})\widehat{\phantom{c}} (\xi)~\neq 0$ only for $|\xi| \le (N+N^\frac12)/3 < N$, so
 $I\left(f_{\lesssim N^{\frac12}}g_{\lesssim N}\right) = f_{\lesssim N^{\frac12}}g_{\lesssim N}$ as well.
 \item $\norm{\mathrm{II}}_{L^2}$ can be written as $\sup_{\norm{h}_{L^2} = 1} \int_{\R^2} h\cdot (\mathrm{II})$. 
 Calling $\widehat f_{\lesssim N^{\frac12}} = a$, $\widehat g_{\gtrsim N} = b$, expanding
 $\mathrm{II}$ in Fourier series and using Plancherel, we have to estimate
 $$ \iint_{\xi_1< N^\frac12/3,\xi_2> N/3} a(\xi_1)b(\xi_2)\Big(m_N(\xi_1)m_N(\xi_2)-m_N(\xi_1+\xi_2)\Big)\hat h(\xi_1+\xi_2)\d\xi_1\d\xi_2.$$
 Using the fact that on the considered interval $m(\xi_1)\equiv 1$ and that, by the mean value theorem, 
 $|m(\xi_1+\xi_2)-m(\xi_1)| \lesssim_s N^{1-s} |\xi_1||\xi_2|^{-2+s}$, we have that 
 \begin{align*}
 &~\int_{\R^2} h\cdot (\mathrm{II}) \\
 =&~\iint_{\xi_1< N^\frac12/3,\xi_2> N/3} a(\xi_1)b(\xi_2)\Big(m_N(\xi_1)m_N(\xi_2)-m_N(\xi_1+\xi_2)\Big)\hat h(\xi_1+\xi_2)\d\xi_1\d\xi_2\\
\lesssim &~N^{(1-s)}\iint_{\xi_1< N^\frac12/2,\xi_2\ge N/2} \left|\frac{a(\xi_1)}{|\xi_1|^{\tilde s+\delta}}\right| \left|\frac{b(\xi_2)|}{|\xi_2|^{\delta}}\right|\frac{|\xi_1|^{1+\tilde s+\delta}}{|\xi_2|^{2-s-\delta}}\left|\hat h(\xi_1+\xi_2)\right|\d\xi_1\d\xi_2\\
\lesssim &~N^{-\left(\frac{1-\tilde s}{2}-\frac32\delta\right)}\iint_{\xi_1< N^\frac12/2,\xi_2\ge N/2} \left|\frac{a(\xi_1)}{|\xi_1|^{\tilde s+\delta}}\right| \left|\frac{b(\xi_2)|}{|\xi_2|^{\delta}}\right|\left|\hat h(\xi_1+\xi_2)\right| \d\xi_1\d\xi_2\\
\lesssim &~N^{-\left(\frac{1-\tilde s}{2}-\frac32\delta\right)} \norm{\frac{a(\xi_1)}{|\xi_1|^{\tilde s+\delta}}}_{L^1} \norm{\frac{b(\xi_2)|}{|\xi_2|^{\delta}}}_{L^2} \norm{\hat h}_{L^2} \\
\lesssim &~N^{-\left(\frac{1-\tilde s}{2}-\frac32\delta\right)} \norm{f}_{H^{1-\tilde s}} \norm{g}_{H^{-\delta}} \norm{h}_{L^2},
\end{align*}
therefore $\norm{\mathrm{II}}_{L^2}\lesssim N^{-\left(\frac{1-\tilde s}{2}-\frac32\delta\right)}\norm{f}_{H^{1-\tilde s}} \norm{g}_{H^{-\delta}}$.
\item By H\"older, Sobolev embeddings and \eqref{I},
\begin{align*}
 \norm{\mathrm{III}}_{L^2} & \lesssim \norm{If_{\gtrsim N^{\frac12}}}_{L^2} \norm{Ig}_{L^\infty} \\
			   & \lesssim_\delta N^{-\frac{1-\tilde s}{2}} \norm{f}_{H^{1-\tilde s}}\norm{Ig}_{W^{3\delta,\delta^{-1}}}\\
			   & \lesssim_\delta N^{-\frac{1-\tilde s}{2}+ 4 \delta} \norm{f}_{H^{1-\tilde s}}\norm{g}_{W^{-\delta,\delta^{-1}}}
\end{align*}
\item By duality (like for $\mathrm{II}$), self-adjointness of $I$, fractional Liebnitz inequality, Sobolev embeddings and \eqref{I}, we have 
\begin{align*}
&~\int h\cdot(\mathrm{IV}) \\
=&~-\int h I(f_{\gtrsim N^{\frac12}}g)\\
= &~-\int I(h) f_{\gtrsim N^\frac12} g\\
\lesssim &~\norm{I(h)f_{\gtrsim N^{\frac12}}}_{W^{2\delta,(1-\delta)^{-1}}} \norm{g}_{W^{-2\delta,\delta^{-1}}}\\
\lesssim &~\begin{multlined}[t] \norm{g}_{W^{-2\delta,\delta^{-1}}} 
			     			\left(\norm{I(h)}_{H^{2\delta}}\norm{f_{\gtrsim N^{\frac12}}}_{L^{\left(\frac12-\delta\right)^{-1}}} \vphantom{\norm{I(h)}_{L^{\left(\frac12-\delta\right)^{-1}}}\norm{f_{\gtrsim N^{\frac12}}}_{L^2}}\right.\\
						\left.\vphantom{\norm{I(h)}_{H^{\delta}}\norm{f_{\gtrsim N^{\frac12}}}_{L^{\left(\frac12-\delta\right)^{-1}}}}+ \norm{I(h)}_{L^{\left(\frac12-\delta\right)^{-1}}}\norm{f_{\gtrsim N^{\frac12}}}_{H^{2\delta}}\right)\end{multlined} \\
\lesssim &~\norm{g}_{W^{-2\delta,\delta^{-1}}}\norm{I(h)}_{H^{2\delta}}\norm{f_{\gtrsim N^{\frac12}}}_{H^{2\delta}} \\
\lesssim &~N^{2\delta} N^{-\frac12\left(1-\tilde s-2\delta\right)} \norm{g}_{W^{-2\delta,\delta^{-1}}} \norm{h}_{L^2}\norm{f_{\gtrsim N^{\frac12}}}_{H^{1-\tilde s}}\\
\lesssim &~N^{-\left(\frac{1-\tilde s}{2}-3\delta\right)}\norm{g}_{W^{-2\delta,\delta^{-1}}} \norm{h}_{L^2}\norm{f_{\gtrsim N^{\frac12}}}_{H^{1-\tilde s}},
\end{align*}
so $\norm{\mathrm{IV}}_{L^2}\lesssim N^{-\left(\frac{1-\tilde s}{2}-3\delta\right)}\norm{g}_{W^{-2\delta,\delta^{-1}}} \norm{f}_{H^{1-\tilde s}}$.
\end{itemize}
Therefore, by choosing $\delta$ such that $\gamma \ge 4\delta$, $\eta = \delta$ and $p = \delta^{-1}$, we obtain \eqref{eq:commII}.
 \end{proof}
\begin{Lemma} \label{commutator}
Let $0<k\le 2$. For every $\gamma > 0$, there exist $p(\gamma) > 1, \eta(\gamma)> 0$ such that
\begin{multline} \label{comm}
\norm{I\big(v^k\rho\wick{\psi^{3-k}}\big) - (Iv)^k I\big(\rho\wick{\psi^{3-k}}\big)}_{L^2} 
\\ \lesssim_{s,\gamma} N^{-\frac{1-k(1-s)}{2} + \gamma} \norm{Iv}_{H^1}^k \norm{\rho\wick{\psi^{3-k}}}_{W^{-\eta(\gamma),1}\cap W^{-\eta(\gamma),p(\gamma)}}.
\end{multline}
\end{Lemma}
\begin{proof}
We have that
\begin{align}
&~\norm{I\big(v^k\rho\wick{\psi^{3-k}}\big) - (Iv)^k I\big(\rho\wick{\psi^{3-k}}\big)}_{L^2} \notag\\
\lesssim &~\norm{I\big(v^k\rho\wick{\psi^{3-k}}\big) -I\big(v^k\big)I\big(\rho\wick{\psi^{3-k}}\big)}_{L^2} \tag{I}\\
+ &~\norm{\left(I\big(v^k\big) - \big(Iv\big)^k\right)I\big(\rho\wick{\psi^{3-k}}\big)}_{L^2} \tag{II}.
\end{align}
\begin{itemize}
\item By Sobolev embeddings and fractional Liebnitz, we have that 
$$ \norm{v^{k}}_{H^{1-k(1-s)}} \lesssim \norm{v^{k}}_{W^{s,\frac{2}{1+(k-1)(1-s)}}}
\lesssim \norm{v}_{H^{s}}\norm{v}_{L^{\frac2{(1-s)}}}^{k-1} \lesssim \norm{v}_{H^{s}}^k.$$
Therefore, by \eqref{eq:commII},
$$\norm{\mathrm{I}}_{L^2} \lesssim N^{-\frac{1-k(1-s)}2 + \gamma} \norm{v}^k_{H^{s}} \norm{\rho\wick{\psi^{3-k}}}_{W^{-\eta(\gamma),p(\gamma)}}. $$
From \eqref{inverseI}, we have that $\norm{v}_{H^{s}} \lesssim \norm{Iv}_{H^1},$ so
$$\norm{\mathrm I}_{L^2} \lesssim_{(1-s),\gamma} N^{-\frac{1-k(1-s)}2 + \gamma}\norm{Iv}_{H^1}^k \norm{\rho\wick{\psi^{3-k}}}_{W^{-\eta(\gamma),p(\gamma)}}. $$
\item From H\"older, \eqref{eq:commI} and Sobolev embeddings, we have 
\begin{align*}
\norm{\mathrm{II}}_{L^2} & \lesssim \norm{I\big(v^k\big) - \big(Iv\big)^k}_{L^2} \norm{I\big(\rho\wick{\psi^{3-k}}\big)}_{L^\infty} \\
	& \lesssim_{s,\delta} N^{-(1-k(1-s))} \norm{Iv}_{H^1}^k \norm{I\big(\rho\wick{\psi^{3-k}}\big)}_{W^{3\delta,\delta^{-1}}} \\
	& \lesssim_{s,\delta} N^{-(1-k(1-s))} N^{4\delta} \norm{Iv}_{H^1}^k \norm{\rho\wick{\psi^{3-k}}}_{W^{-\delta,\delta^{-1}}}.
\end{align*}
Choosing $\delta$ small enough, we have that $1-k(1-s) - 4\delta > \frac{1-k(1-s)}2 - \gamma$, so the main contribution comes from $\mathrm{I}$. We get \eqref{comm} by taking $\gamma' =\ge 4\delta$, 
$p(\gamma') = \max(p(\gamma),\delta^{-1})$, $\eta(\gamma') = \max(\eta(\gamma),\delta)$, and then renaming $\gamma = \gamma'$.
\end{itemize}
\end{proof}
\begin{Lemma}\label{worstestimatelemma}
There exists $c>0$ such that for every $0< \eta < \frac18$, 
\begin{equation} \label{worstestimate}
\int_{t_0}^T\eqref{worst}(s)\d s
\lesssim \left(1 + \int_{t_0}^T E^{1 + c \eta}\right) \norm{I \psi}_{L^{\eta^{-1}}_{t,x}}
\end{equation}
\end{Lemma}
\begin{proof}
 Let $0< \theta < 1$. Since we have 
 \begin{align*}
  \norm{Iv}_{H^1} & \lesssim E^\frac12 \\
  \norm{Iv}_{L^4} & \lesssim E^\frac 14,
 \end{align*}
 by Gagliardo-Niremberg we have $\norm{Iv}_{W^{\theta, \frac 4 {1+\theta}}} \lesssim E^{\frac{1+\theta}4}.$
 Therefore, by Sobolev inequality, we have that $\norm{Iv}_{L^{\frac 4 {1-\theta}}} \lesssim E^{\frac{1+\theta}4}$,
 and the implicit constant is uniform in $\theta$ as long as $0\le \theta \le \theta_{\max} <1$. Take 
 $\theta = 4 \eta$. Therefore, by H\"older, 
 \begin{align*}
   \left|\int_{t_0}^T \int_{\T^2} Iv_t (Iv)^2 I \psi\right| 
   &\le \int_{t_0}^T\norm{Iv_t}_{L^2} \norm{Iv}_{L^4} \norm{Iv}_{L^{\frac 4 {1-4\eta}}} \norm{I \psi}_{L_x^{\eta^{-1}}} \\
   & \lesssim \int_{t_0}^T E^\frac12 E^\frac14 (E^\frac{1+4\eta}4)\norm{I \psi}_{L_x^{\eta^{-1}}} \\
   & \lesssim \int_{t_0}^T \left(E^{1+\eta}\right)\norm{I \psi}_{L_x^{\eta^{-1}}} \\
   & \lesssim \left(\int_{t_0}^T \left(E^{1 + \eta}\right)^{\frac 1 {1-\eta}}\right)^{1-\eta} \norm{I \psi}_{L_{x,t}^{\eta^{-1}}} \\
   & \lesssim \left(1 + \int_{t_0}^T E^{\frac{1 + \eta}{1-\eta}}\right)\norm{I \psi}_{L_{x,t}^{\eta^{-1}}}.
 \end{align*}
Therefore, choosing $c = \max_{\eta \in [0, \frac18]} \eta^{-1} \left(\frac{1 + \eta}{1-\eta} - 1\right)$, we have
 $$ \left|\int_{t_0}^T \int_{\T^2} Iv_t (Iv)^2 I \psi\right| \lesssim \left(1 + \int_{t_0}^T  E^{1+c\eta}\right)\norm{I \psi}_{L_{x,t}^{\eta^{-1}}}, $$
 which gives \eqref{worstestimate}.
\end{proof}
\begin{Lemma}\label{Gronwall0}
Let $T>0$. For every $|t - t_0| \le T$, for every $0<\eta\le \frac18$, we have that for every $\gamma > 0$,
\begin{align} 
&~E(t) - E(t_0) \notag\\
\lesssim_{s, \gamma, T}&~\left(1 + \int_{t_0}^t E^{1 + c \eta}\right) \norm{I \psi}_{L^{\eta^{-1}}_{t,x}} \label{worsti} \\
+&~\int_{t_0}^t N^{-(1-3(1-s))} E^2\label{commIi}\\
+&~\sum_{k=0}^2 \int_{t_0}^t N^{-\frac{1-k(1-s)}2 + \gamma}E^\frac {k+1} 2 \norm{\wick{\psi^{3-k}}}_{L^\infty_tW^{-\eta(\gamma),1}\cap W^{-\eta_k(\gamma), p(\gamma) }}\label{commIIi} \\
+&~\int_{t_0}^t N^\gamma \Big(E^\frac34 \norm{\rho\wick{\psi^2}}_{W^{-\gamma,4}} + E^\frac12 \norm{\rho\wick{\psi^3}}_{H^{-\gamma}} \Big) \label{tamei}\\ 
+&~\int_{t_0}^t E(s)\d s \label{harmlessi},
\end{align}
where $c$ is the one given by Lemma \ref{worstestimatelemma} and $\eta(\gamma), p(\gamma)$ are the ones given by Lemma \ref{commutator}.
\end{Lemma}
\begin{proof}
By Lemma \ref{timeDerivative},  
\begin{equation*}
E(t) - E(t_0) = \int_{t_0}^t \eqref{worst}(s) + \eqref{tame}(s) + \eqref{commutators}(s) + \eqref{harmless}(s) \d s.
\end{equation*}
We have that 
\begin{itemize}
\item From \eqref{worstestimate}, $\int_{t_0}^t \eqref{worst}(s) \d s\lesssim \eqref{worsti}$,
\item From \eqref{tameestimate}, $\int_{t_0}^t \eqref{tame}(s) \d s\lesssim \eqref{tamei}$,
\item From \eqref{eq:commI} for the fist term and \eqref{comm} for the second and third term respectively, and H\"older inequality,
$\int_{t_0}^t  \eqref{commutators}(s) \d s \lesssim \eqref{commIi} + \eqref{commIIi}$
\item From \eqref{harmlessestimate}, $\int_{t_0}^t \eqref{harmless}(s) \d s \lesssim \eqref{harmlessi}$.
\end{itemize}
\end{proof}
\begin{Lemma}\label{tau}
Let $T>0$, and let 
\begin{equation}
A(N) = \frac{\norm{I_N\rho\psi}_{L^{\log N}([-T,T]\times\R^2)}}{ \log N}.
\end{equation}
For $\gamma > 0$, $M\ge 1$,$\Lambda\ge1$, define 
\begin{gather}
\Omega_M^\gamma := \left\{ \max_k \norm{\wick{\rho \psi^{3-k}}}_{L^\infty_t([-T,T];W^{-(\max(\eta_k(\gamma),\gamma)), \max(p_k(\gamma),2)})} \le M \right\},\label{Omegagamma}\\
\Omega_\Lambda(N) := \left\{ A(N) \le \Lambda\right\}\label{OmegaLambda}.
\end{gather}
Then for $\gamma=\gamma(s)$ small enough, $\alpha < 1 - 3(1-s)$, $\delta< \beta < \alpha$, $\omega \in \Omega^\gamma_M$, there exists $\tau = \tau(s,M,\Lambda,\alpha-\beta)$
such that if $\omega \in \Omega_M^\gamma\cap\Omega_{\Lambda}(N)$, $E(t_0) \le N^\beta/2$, $|t_0| \le T$, $N\ge N_0 = N_0(s,T,M,\Lambda)$, then $E(t) \le N^\alpha$ for every $t$ such that $|t|\le T$ and 
$|t-t_0| \le \tau$.
\end{Lemma}
\begin{proof}
By Lemma \ref{Gronwall0}, as long as $E\le N^\alpha$, since $\alpha < 1 - 3(1-s)$, for $N$ big enough we have that 
$\eqref{commIi} + \eqref{commIIi} \le 1+\int_{t_0}^T E$. Similarly, from Young's inequality, for some universal constant $C$, we have
$$\eqref{tamei} \le \int_{t_0}^T E(s) \d s + C N^{4\gamma}M^4.$$
Choosing $\eta = (\log N)^{-1}$ in \eqref{worsti}, as long as $E \le N^\alpha$, we get
$$\eqref{worsti} \le A(N)\log N\Big(1 + \int_{t_0}^T E^{1 + c(\log N)^{-1}}(s) \d s\Big) \le \Lambda\log N\Big(1 + e^{c\alpha} \int_{t_0}^T E(s) \d s\Big).$$
Therefore, as long as $E(t) \le N^\alpha$, for $N$ big enough (depending on $s,T,M,\Lambda$),
\begin{align}
&~E(t) \notag\\
\le &~E(t_0) + C(s,\gamma,T)\Lambda\log N\Big(1 + \int_{t_0}^T E(s) \d s\Big)\notag\\
&~+1 + \int_{t_0}^T E(s) \d s + C N^{4\gamma}M^4 +  \int_{t_0}^T E(s) \d s \notag\\
\le &~\frac12N^\beta + C(s,\gamma,T)\Lambda\log N + C'N^{4\gamma}M^4 \notag\\
&~+ \Big(2 + C(s,\gamma,T)\Lambda\log N\Big) \int_{t_0}^T E(s) \d s.\notag
\\
\le&~ N^\beta + C'(s,\gamma,T,\Lambda)\log N\int_{t_0}^T E(s) \d s. \label{hopeful}
\end{align}
Let $\bar t = \max\{s: t_0\le s \le T, E(s) \le N^\alpha\}$, $\tilde t = \min\{s: t_0\ge s \ge -T, E(s) \le N^\alpha\}$. 
Then the lemma is proven if we show that if $\bar t \neq T$, then $|\bar t-t_0| \ge \tau(s,\gamma,T,\Lambda)$ and similarly if $\tilde t \neq -T$, then $|\tilde t-t_0| \ge \tau(s,\gamma,T,\Lambda)$.
For $\tilde t\le t \le \bar t$, by definition, \eqref{hopeful} holds, 
so by Gronwall
\begin{equation} \label{Gronwall}
E(t) \le N^\beta\exp\big(|t-t_0|C'(s,\gamma,T,\Lambda)\log N\big).
\end{equation}
Suppose that $\bar t \neq T$. Then one must have $E(\bar t) = N^\alpha$. Therefore, by \eqref{Gronwall}, $|\bar t - t_0| \ge \tau:= \frac{(\alpha-\beta)}{C'(s,\gamma,T,\Lambda)}$. The same holds for $\tilde t$, and the lemma is proven.
\end{proof}
\begin{proof}[Proof of Proposition \ref{lgwp}]
Let $\epsilon > 0$, $T>0$, let $2(1-s) < \beta < \alpha < 1-3(1-s)$, and let $\gamma$ as in Lemma \ref{tau}.\\
By Proposition \ref{psi},$\mathrm{(iv)}$, $\norm{\wick{\rho \psi^{3-k}}}_{L^\infty_t([-T,T];W^{-(\max(\eta_k(\gamma),\gamma)), \max(p_k(\gamma),2)})} < +\infty$ a.s., 
so there exists $M=M(T,\epsilon)$ such that $\P(\Omega_M^\gamma) \ge 1-\frac\epsilon2$, where $\Omega_M^\gamma$ is defined in \eqref{Omegagamma}. Moreover, by Proposition \ref{psi},$(\mathrm{v})$, 
$$P(\Omega_\Lambda(N)^c) \le C_{T,\rho,m}^{\log N} \Lambda^{-\log N} = N^{\log C_{T,\rho,m} - \log \Lambda}.$$ If $\Lambda = C_{T,\rho,m}e^2$, this expression is summable in $N$, so for $N$ big enough, $N\ge \tilde N = \tilde N(\epsilon)$, \linebreak$P(\bigcap_{N \ge \tilde N} \Omega_{\Lambda}(N)) \ge 1 - \frac \epsilon2$. Let 
$\Omega_{T,\epsilon} := \Omega_M^\gamma \cap \bigcap_{N \ge \tilde N} \Omega_{\Lambda}(N)$. By inclusion-exclusion, we have that $\P(\Omega_{T,\epsilon}) \ge 1 - \epsilon$.

For this choice of $M$, $\Lambda$, $\gamma$, $\alpha, \beta$, let $N_0$ and $\tau$ be the ones given by Proposition \ref{tau}, and take $(u_0, u_1) \in \H^s$. Define a sequence $N_k$ of integers recursively. Take $N_1$
such that $N_1 \ge \max(N_0,\tilde N)$ and 
$$N_1^{2(1-s)} \norm{(u_0,u_1)}_{\H^s}^2 + \norm{u_0}_{H^s}^4 \ll N_1^\beta.$$
By Sobolev embeddings, \eqref{I} and \eqref{inverseI},
\begin{equation}\label{inductionstep}
\begin{aligned}
E(I_{N_1}v(0),I_{N_1}v_t(0)) &= E(I_{N_1}u_0,I_{N_1}u_1) \\
&\lesssim \norm{(I_{N_1}u_0,I_{N_1}u_1)}_{H^1}^2 + \norm{I_{N_1}u_0}_{L^4}^4  \\
&\lesssim N_1^{2(1-s)} \norm{(u_0,u_1)}_{H^s}^2 + \norm{I_{N_1}u_0}_{H^s}^4 \\
&\lesssim N_1^{2(1-s)} \norm{(u_0,u_1)}_{H^s}^2 + \norm{u_0}_{H^s}^4, \\
\end{aligned}
\end{equation}
so we will have $E(I_{N_1}v(t_0),I_{N_1}v_t(t_0)) \le \frac12 N_1^\beta$, therefore by Lemma \ref{tau} one has that 
$\norm{(v(t),v_t(t))}_{\H^s}^2 \lesssim E(I_{N_1}v(t),I_{N_1}v_t(t)) \le N^\alpha$ for $t\le T$, $t_0 \le t\le t_0 + \tau$, and similarly backwards in time. \\
Then take $N_{k+1} \gg N_k$ such that 
\begin{equation}
N_{k+1}^{2(1-s)} N_k^\alpha + N_k^{2\alpha} \ll N_{k+1}^\beta.
\end{equation}
If one has 
\begin{equation}\label{induction}
E(I_{N_k}v(t_0 + (k-1)\tau),I_{N_k}v_t(t_0+(k-1)\tau)) \le \frac12 N_k^\beta \text{ if }t_0 + (k-1)\tau \le T,
\end{equation}
by Lemma \ref{tau}, 
\begin{equation}
\begin{aligned}
&~\norm{(v,v_t)}^2_{L^\infty([t_0 + (k-1)\tau, \min(t_0 + k\tau, T)];\H^s)} \\
\lesssim &~\sup_{t_0 + (k-1)\tau \le s\le \min(t_0 + k\tau, T)} E(I_{N_k}v(s),I_{N_k}v_t(s)) \le N_k^\alpha
\end{aligned}
\end{equation}
and similarly backwards in time. Therefore, because of the blowup condition \eqref{eq:blowup}, Propostion \ref{lgwp} is shown if we show \eqref{induction}. Proceeding inductively, we know \eqref{induction} for $N_1$, and proceeding as in \eqref{inductionstep},
\begin{align*}
&~E(I_{N_{k+1}}v(t_0 + k\tau),I_{N_k}v_t(t_0+k\tau)) \\
\lesssim &~N_{k+1}^{2(1-s)} \norm{(v(t_0 + k\tau),v_t(t_0 + k\tau))}_{H^s}^2 + \norm{v(t_0 + k\tau)}_{H^s}^4\\
\lesssim  &~N_{k+1}^{2(1-s)}E(I_{N_k}v(t_0 + k\tau),I_{N_k}v_t(t_0 + k\tau)) + E(I_{N_k}v(t_0 + k\tau),I_{N_k}v_t(t_0 + k\tau))^2 \\
\lesssim &~N_{k+1}^{2(1-s)}N_k^\alpha + N_k^{2\alpha} \ll N_{k+1}^\beta,
\end{align*}
so $E(I_{N_{k+1}}v(t_0 + k\tau),I_{N_k}v_t(t_0+k\tau)) \le \frac12N_{k+1}^\beta$ and we have \eqref{induction}.
\end{proof}

\section{Independence from the cutoff and global well-posedeness for the global equation}
In this section, we prove that on appropriate space-time regions, the solution to \eqref{LSNLW} does not depend on the particular choice of the cutoff $\rho$, 
and proceed to the proof of Theorem \ref{main}.
\begin{Proposition}[Finite speed of propagation for SNLW] \label{finitespeed}
Let $R,T>~0$, $x_0 \in \R^2$, $t_0 \in \R$. Let $u_1, u_2$ be solutions to \eqref{SNLW} on $B(x_0,R)$ for a time $T$,
in the sense that $u_j|_{B(x_0,R)} = \psi|_{B(x_0,R)} + v_j|_{B(x_0,R)}$, $v_j \in C([t_0-T,t_0+T];H^s_\loc)$, $s > \frac23$, and 
\begin{multline}\label{SNLWLocal}
v_j(t)|_{B(x_0,R)} = \left( \cos((t-t_0)|\nabla|) v_j(t_0) \vphantom{\int_{t_0}^t}\ + \frac{\sin((t-t_0)|\nabla|)}{|\nabla|} 
\partial_t v_j(t_0) \right. \\ \left. \left.+ \int_{t_0}^t \frac{\sin{((t-t')|\nabla|)}}{|\nabla|} \wick{(\psi+v_j)^3}(t')\d t'\right)\right|_{B(x_0,R)}
\end{multline}
for every $|t-t_0| \le T$.
Suppose moreover that $v_1(0)=v_2(0)$, $\partial_tv_1(0) = \partial_tv_2(0)$ on $B(x_0,R)$ and $v_j \in C([t_0-T,t_0+T]; H^s(\R^2))$, $s>\frac 12$. Then $v_1(t)|_{B(x_0,R-|t-t_0|)} 
= v_2(t)|_{B(x_0,R-|t-t_0|)}$ for every $|t-t_0|\le T$.
\end{Proposition}
\begin{proof}
Without loss of generality, assume that $t_0 = 0$, $x_0 = 0$.
Let $D_t = B_{R-|t|}$. Recalling that the kernels of $\cos(s|\nabla|), \frac{\sin(s|\nabla|)}{|\nabla|}$
are distributions supported in $B_{s}$, from \eqref{SNLW} we have that
$$ v_1-v_2|_{D_t} = \int_0^t \frac{\sin{((t-t')|\nabla|)}}{|\nabla|} \left.\left[(v_1-v_2)(3\wick{\psi^2} + 3\psi(v_1+v_2) + v_1^2+v_2^2 + v_1v_2)\right]\right|_{D_{t'}}\d t'.$$
Proceeding as in Proposition \ref{LWP}, we obtain that 
\begin{align*}
&\phantom{\lesssim} \norm{v_1(t)-v_2(t)}_{H^s(D_t)}\\ 
&\begin{multlined}\lesssim_R \int_0^t \norm{v_1-v_2}_{H^s(D_t')}
(1+ \norm{v_1}_{C([-R,R];H^s(B_R))}^2  + \norm{v_2}_{C([-R,R];H^s(B_R))}^2)\\ \times \max_{1\le j \le 3} \norm{\rho\wick{\psi^j}}_{L^1_tW^{0-,6+}_x}
\end{multlined}
\end{align*}
for any $\rho \in C_c^\infty$ with $\rho \equiv 1$ on $B_R$.
Therefore by Gronwall, 
$$ \norm{v_1(t)-v_2(t)}_{H^s(D_t)} \le C(R,v_1,v_2,\omega) \norm{v_1(0)-v_2(0)}_{H^s(B_R)}= 0.$$
\end{proof}
From this proposition, we obtain immediately the following:
\begin{Corollary} \label{constancy}
Let $T>0$, let $t_0=0$, and let $\rho_1$, $\rho_2$ be two cutoff functions such that $\rho_1(x) = \rho_2(x) = 1$ for
every $x \in B_{2T}$. Let $s>\frac45$. Let $(u_0,u_1)\in \H^s_\loc$, and let $v_1,v_2$ be respectively the solutions to 
\eqref{LSNLW} with cutoff function $\rho_1$ and $\rho_2$ and initial data respectively $(\rho_1u_0,\rho_1u_1)$
and $(\rho_2u_0,\rho_2u_1)$. Then $v_1(t,x) = v_2(t,x)$ for every $|x|,|t| < T$.
\end{Corollary}
\begin{proof}[Proof of Theorem \ref{main}]
Let $\rho_n$ be a cutoff function such that $\rho_N(x) = 1$ for every $|x| \le n$. 
Let $(u_0,u_1) \in \H^s_\loc$, and let $v_n$ be the solution of \eqref{LSNLW} with cutoff function $\rho_n$ and initial data $(\rho_nu_0,\rho_nu_1)$. By Proposition \ref{lgwp}, we will have $v_n\in C(\R;\H^s)$. By Corollary \ref{constancy},  
$$v:= \lim_{n\to +\infty} v_n$$
is well defined, and we have that $v|_{[-T,T] \times B_R} = v_{\lceil T\vee R\rceil}|_{[-T,T] \times B_R}$.
By Proposition \ref{lgwp} again, $v$ will also be a continuous function of $(u_0,u_1)$ with 
values in $C([-T,T];\H^s_\loc)$. Therefore the theorem is proven if we show that every solution $\tilde u = \psi + \tilde v$ of \eqref{SNLW} with $v \in C([-T,T];\H^s_\loc)$ satisfies $\tilde v(t) = v(t)$ for every $t \le T$. Let $\phi \in C^\infty_c((-T,T) \times \R^2)$ be a test function. Let $n\in \N$ be such that $\supp(\phi) \subseteq [-n,n] \times B_n$.
By Proposition \ref{finitespeed}, we have that 
$$\tilde v|_{[-n,n] \times B_n} = v_n|_{[-n,n] \times B_n} = v|_{[-n,n] \times B_n}.$$
Therefore, $\dual{\tilde v}{\phi} = \dual{v}{\phi}$, so $\tilde v = v$ as space-time distributions. Since they both 
belong to the space $C([-T,T];\H^s_\loc)$, the equality must hold in the space $C([-T,T];\H^s_\loc)$ as well, hence $\tilde v(t) = v(t)$ for every $t\le T$.
\end{proof}

\end{document}